\documentclass[10pt]{amsproc}

\title[]{The Morita equivalence between parametrized spectra and module spectra}
\author{John A. Lind}
\address{Reed College, Portland, Oregon 97202}
\email{john.alexander.lind@gmail.com}
\author{Cary Malkiewich}
\address{Binghamton University, Binghamton, New York 13902}
\email{malkiewich@math.binghamton.edu}

\usepackage{amsmath}
\usepackage{amssymb}
\usepackage{amsthm}
\usepackage{amsrefs}
\usepackage{enumerate} %% easy control of enumerate items
\usepackage{mathrsfs}
\usepackage{stmaryrd}
\usepackage[all]{xy}
\usepackage[mathcal]{eucal}
\usepackage{verbatim}  %%includes comment environment
\usepackage{leftidx}

 \newdir{ >}{{}*!/-5pt/@{>}}  %%makes arrows of the form >----> look right in xypic.  use @{ >->} to produce these arrows

\numberwithin{equation}{section}

\newtheorem{theorem}[equation]{Theorem}

\newtheorem{proposition}[equation]{Proposition}

\newtheorem{lemma}[equation]{Lemma}
\newtheorem{corollary}[equation]{Corollary}

\theoremstyle{definition}
\newtheorem{definition}[equation]{Definition}
\newtheorem{remark}[equation]{Remark}

\newtheorem{example}[equation]{Example}

\newcommand{\bL}{\mathbf{L}}\newcommand{\bR}{\mathbf{R}}\newcommand{\bZ}{\mathbf{Z}}

\newcommand{\cB}{\mathcal{B}}\newcommand{\cC}{\mathcal{C}}\newcommand{\cD}{\mathcal{D}}\newcommand{\cE}{\mathcal{E}}\newcommand{\cG}{\mathcal{G}}\newcommand{\cK}{\mathcal{K}}\newcommand{\cM}{\mathcal{M}}\newcommand{\cN}{\mathcal{N}}\newcommand{\cS}{\mathcal{S}}\newcommand{\cT}{\mathcal{T}}

\DeclareMathOperator{\id}{id}
\DeclareMathOperator{\eval}{eval}
\DeclareMathOperator{\Fun}{Fun}

\DeclareMathOperator*{\hocolim}{hocolim}

\DeclareMathOperator{\ho}{ho}
\DeclareMathOperator{\ob}{ob}

\DeclareMathOperator{\Map}{Map}

\renewcommand{\phi}{\varphi}
\renewcommand{\emptyset}{\O}
\providecommand{\sma}{\wedge}
\DeclareMathOperator{\osma}{\overline{\wedge}}

\providecommand{\arr}{\longrightarrow}
\providecommand{\oarr}[1]{\overset{ #1 }{\longrightarrow}}

\providecommand{\abs}[1]{\lvert #1 \rvert}

\providecommand{\THH}{\mathrm{THH}}

\newcommand{\op}{\textup{op}}

\newcommand{\conn}{\mathrm{conn}}

\providecommand{\Sp}{\cS p}
\providecommand{\Ex}{\cE\mathrm{x}}
\providecommand{\Top}{\cT\mathrm{op}}
\providecommand{\TopGrp}{\cT\mathrm{op}\cG\mathrm{rp}}
\providecommand{\Bimod}{\cB\mathrm{imod}}
\providecommand{\coev}{\mathrm{coev}}
\providecommand{\eval}{\mathrm{eval}}

\providecommand{\Mod}{\cM\mathrm{od}}
\providecommand{\sSet}{s\mathrm{Set}}

\makeatletter
\newsavebox{\@brx}
\newcommand{\llangle}[1][]{\savebox{\@brx}{\(\m@th{#1\langle}\)}%
  \mathopen{\copy\@brx\kern-0.5\wd\@brx\usebox{\@brx}}}
\newcommand{\rrangle}[1][]{\savebox{\@brx}{\(\m@th{#1\rangle}\)}%
  \mathclose{\copy\@brx\kern-0.5\wd\@brx\usebox{\@brx}}}
\makeatother

\providecommand{\qI}{\mathbin{^{q}\!I}}
\providecommand{\qfI}{\mathbin{^{qf}\!I}}
\providecommand{\qJ}{\mathbin{^{q}\!J}}
\providecommand{\qfJ}{\mathbin{^{qf}\!J}}

\begin{document}

\date{February 24, 2017, and in revised form, August 30, 2017}

\maketitle

\begin{abstract}
We give a Quillen equivalence between May and Sigurdsson's model category of parametrized spectra over $BG$, and Mandell, May, Schwede, and Shipley's model category of modules over the orthogonal ring spectrum $\Sigma^\infty_+ G$, for each topological group $G$. More generally, for a topological category $\cC$ we introduce an ``aggregate'' model structure on the category of diagrams of spectra indexed by $\cC$, and prove that it is Quillen equivalent to spectra over $B\cC$. This strengthens earlier results that hold at the level of $\infty$-categories, and allows us to give a simple argument that the ``dualizable'' parametrized spectra are precisely the homotopy retracts of the finite cell spectra.
\end{abstract}

\tableofcontents

\section{Introduction}

The theory of parametrized spectra is of central importance to the geometric applications of homotopy theory, and to the construction of transfer maps and Thom spectra. The homotopy category of parametrized spectra over a fixed base space $B$ was first defined in the work of Clapp and Puppe \citelist{\cite{Clapp} \cite{Clapp_Puppe}}. May and Sigurdsson constructed a Quillen model category with this homotopy category \cite{MS}, and shortly thereafter Ando, Blumberg, and Gepner gave an elementary construction of its $\infty$-category \cite{ABG}.

The model category of May and Sigurdsson suffers a few deficiencies, most notably that it is not known to be a monoidal model category.
Its compact objects also evaded a simple geometric characterization, primarily because of difficulties with fibrant replacement. However, since it is a stable model category, the work of Schwede and Shipley \cite{SS} shows that it is Quillen equivalent to a category of modules over some ring spectrum $R$. The ring spectrum $R$ produced by their theorem may be rather unwieldy, but it has an easy characterization of its compact objects as retracts of finite cell objects, and when defining monoidal structures on the category $R$-modules one does not encounter the same difficulties as above. It is therefore desirable to find a simple model for $R$ and to make this ``Morita equivalence'' explicit, so that we may carry the clear insights from module theory back over to the geometric setting of parametrized spectra.  The present paper does exactly that, and although the theorem of Schwede and Shipley is not invoked in our proof, its perspective guides our arguments.

We work with a connected base space, which we may assume is the classifying space $BG$ of a topological group, or more generally a group-like topological monoid. We provide an explicit, highly structured equivalence between spectra parametrized over $BG$ and modules over the suspension spectrum $\Sigma^{\infty}_{+} G$ of $G$:
\begin{theorem}\label{thm:intro_quillen_equiv_G}
For each well-based group-like topological monoid $G$, there is a Quillen equivalence between the category $\Sp_{BG}$ of parametrized spectra over $BG$ and the category $\Mod_{\Sigma^{\infty}_{+} G}$ of modules over the ring spectrum $\Sigma^{\infty}_{+}G$.
\end{theorem}

More generally, for any topological category $\cC$, there is a Quillen equivalence between parametrized spectra over $B\cC$ and a certain localization of the category of $\cC$-diagrams of spectra, where the fibrant objects are diagrams in which every morphism is a stable equivalence. When $\cC^{\op}$ is the category of open sets in a sufficiently nice open cover of a space $X$, these diagrams are closely related to locally constant homotopy sheaves on $X$, see \S\ref{sec:groupoid}.

Theorem \ref{thm:intro_quillen_equiv_G} lifts the earlier result that the underlying $\infty$-categories of $\Sigma^{\infty}_{+} G$-modules and of spectra over $BG$ are equivalent \cite{ABG}. It allows for geometric arguments to be directly transferred from the setting of module spectra to parametrized spectra, e.g. Theorem \ref{thm:intro_cwdualizable} below.

The broader significance of Theorem \ref{thm:intro_quillen_equiv_G} is that it leads to a zig-zag of Quillen equivalences between May and Sigurdsson's model category and a combinatorial model category (using for instance the forgetful functor to symmetric module spectra on simplicial sets, or \cite[2.4.2]{harpaz_nuiten_prasma}). It has been pointed out to the authors that by the framework of Dugger \cite{dugger_universal}, any equivalence between the underlying $\infty$-category of the May-Sigurdsson model category and an $\infty$-category arising from a combinatorial model category can consequently be lifted to a zig-zag of Quillen equivalences. Therefore, in essence, all model categories of parametrized spectra are equivalent. The model category of $\Sigma^\infty_+ BG$-comodule spectra defined by Hess and Shipley is one such example, by Theorem \ref{thm:intro_quillen_equiv_G} and \cite[5.4]{HS}.

Since modules over $\Sigma^{\infty}_{+} G$ are the same thing as spectra with an action of $G$, Theorem \ref{thm:intro_quillen_equiv_G} is an instantiation of an old meta-theorem: objects with an action of $G$ are equivalent to families of such objects over $BG$. This meta-theorem has appeared in various forms, e.g. \citelist{ \cite{ABG} \cite{DDK} \cite{Sh08}}, but really goes all the way back to the classical theory of fiber bundles.  From a different point of view, this result is another instantiation of the philosophy that parametrized objects are equivalent to objects equipped with an action of a category $\cC$. This is at the heart of Lurie's straightening/unstraightening theorem \cite{HTT}*{2.2.1.2}, and indeed the work of Ando-Blumberg-Gepner on the $\infty$-category of parametrized spectra \cite{ABG} takes this as its starting point.  Finally, we note that a parametrized spectrum over $BG$ with a specified fiber $M$ is classified by a map $BG \arr B\mathrm{hAut}(M)$ by the work of the first author \cite{bundles_spectra}.  We can think of the classifying map as encoding a homotopy coherent action of $G$ on $M$.  Theorem \ref{thm:intro_quillen_equiv_G} is a more rigid version of the classification theorem that provides a strict action.

Theorem \ref{thm:intro_quillen_equiv_G} arose as part of our recent work \cite{LM} relating the transfer map in topological Hochschild homology (THH) and the refined Reidemeister trace.
%, in which we resolve an open question about the $A$-theory transfer, and provide a model for the THH transfer that allows explicit homology calculations.
It is essential to our applications that the Quillen equivalence of Theorem \ref{thm:intro_quillen_equiv_G} respects the derived base change functors $f_{!}, f^*, f_{*}$ on either side of the equivalence.  We use the language of indexed symmetric monoidal categories \cite{PS12} to encode this agreement, whose precise statement takes the following form.

\begin{theorem}\label{thm:intro_indexed_equiv}  Let $\Top_{\ast}^{\conn}$ denote the cartesian monoidal category of pointed connected topological spaces.  Upon passage to homotopy categories, the right Quillen adjoint of Theorem \ref{thm:intro_quillen_equiv_G} gives an equivalence of $\Top_{\ast}^{\mathrm{conn}}$-indexed symmetric monoidal categories $\ho \Sp_{(-)} \simeq \ho \Mod_{\Sigma^{\infty}_{+}\Omega (-)}$
\end{theorem}

Theorem \ref{thm:intro_indexed_equiv} gives a succinct proof of the well-known equivalence $\THH(\Sigma^{\infty}_{+} \Omega X) \simeq \Sigma^{\infty}_{+} LX$, where $LX = \Map(S^1, X)$ denotes the free loop space \cite{madsen_survey}. This is because it allows us to identify the derived tensor product
\[ \Sigma^\infty_+ \Omega X \sma^{\bL}_{\Sigma^\infty_+ \Omega X \sma \Sigma^\infty_+ \Omega X^\op} \Sigma^\infty_+ \Omega X \]
and the suspension spectrum of the homotopy pullback $X \times^{h}_{X \times X} X \simeq LX$ as instances of the same recipe, described in terms of base-change functors (see Remark \ref{rem:free_loop}).
%In fact, there is an equivalence of this form for any parametrized spectrum over $X \times X$, the case of the free loop space arising from the stable homotopy type of the diagonal $\Delta \colon X \arr X \times X$.

Lurking in the background of this discussion is the analogy between the bicategory $\Bimod$ of ring spectra, bimodules, and homotopy classes of bimodule maps, and the bicategory $\Ex$ of spaces, parametrized spectra and homotopy classes of maps of parametrized spectra \cite{MS}*{\S17}.  Another application of our Morita equivalence is to characterize the duality theory internal to the bicategory $\Ex$, known as Costenoble-Waner duality, using the duality theory of bimodules.  In particular, we answer a question posed by May and Sigurdsson \cite{MS}*{\S18.2}, characterizing the compact objects in the triangulated category $\ho \Sp_{B}$ as homotopy retracts of finite cell spectra.

\begin{theorem}\label{thm:intro_cwdualizable}
A parametrized spectrum $X$ over $B$ is Costenoble-Waner dualizable precisely when it is a retract in the homotopy category of a finite cell spectrum over $B$.  More generally, a 1-cell $X$ in the bicategory $\Ex$ from $A$ to $B$ is dualizable over $B$ precisely when each derived pullback $i_{a}^* X$ to a point $a \in A$ is a retract of a finite cell spectrum over $B$.
\end{theorem}

The equivalence of underlying $\infty$-categories could of course be used to get an analogous result, but the Quillen equivalence of Theorem \ref{thm:intro_indexed_equiv} allows for a pleasingly direct argument. Finally, we get as a corollary a new interpretation of Waldhausen's algebraic $K$-theory of spaces.

\begin{corollary}\label{cor}
Waldhausen's functor $A(B)$ is equivalent to the $K$-theory of the category of Costenoble-Waner dualizable spectra over $B$.
\end{corollary}

\subsection*{Outline} Our strategy to prove Theorem \ref{thm:intro_quillen_equiv_G} applies with $G$ replaced by a topological category $\cC$, and our arguments proceed with that level of generality.  In \S\ref{sec:cells}, we discuss the philosophy of cell complexes, explain the distinction between $qf$- and $q$-cells in parametrized homotopy theory, and recall the May-Sigurdsson model structure on the category of parametrized spectra.  In \S\ref{sec:parametrized_intro}, we describe the category $\cC\Sp$ of $\cC$-diagrams of spectra, equipped with the projective model structure, and construct a Quillen adjunction $(\beta, \Phi)$ between $\cC$-spectra and parametrized spectra over the classifying space $B\cC$.  In \S\ref{sec:aggregate}, we construct the \emph{aggregate} model structure on $\cC\Sp$ and prove that the adjunction $(\beta, \Phi)$ is a Quillen equivalence between it and the stable model structure on $\Sp_{B\cC}$.  When $\cC$ is a groupoid up to homotopy, in particular a group or group-like monoid, the aggregate model structure coincides with the projective model structure, as we prove in \S\ref{sec:groupoid}.  We then deduce Theorem \ref{thm:intro_quillen_equiv_G} as a consequence.  In \S\ref{sec:indexed}, we set up the language of indexed monoidal categories and prove Theorem \ref{thm:intro_indexed_equiv}.  Finally, in \S\ref{sec:cw_duality}, we derive the results in the introduction regarding duality and finite cell spectra.

\subsection*{Conventions}  By a space, we mean a compactly generated topological space, and we denote the category of spaces by $\Top$.  In May-Sigurdsson's work on parametrized homotopy theory \cite{MS}, they work more generally with the category $\cK$ of $k$-spaces because the formation of internal-hom objects in the category of parametrized spaces constructs spaces that are not compactly generated.  We do not use internal-hom objects in this paper, and we may safely work in $\Top$.  Alternatively, this paper may be read as taking place in $\cK$ instead.

\subsection*{Acknowledgements}

The authors thank Thomas Nikolaus for a helpful discussion of the consequences of Theorem \ref{thm:intro_indexed_equiv} for $\infty$-categories, and of an alternate approach to Theorem \ref{thm:intro_cwdualizable}. J.L. was supported in part by the DFG through SFB1085. C.M. was supported in part by an AMS Simons Travel Grant.

% Conceptual Galois correspondence: points of BG are G-orbits, if X is a free G-space it's comprised of G-orbits, each of which is a point over BG. This can be made more precise by factoring the equivalence through the category of G-spaces over EG, which is literally equivalent to spaces over BG. Another cool observation is $F_{BG}(EG,Y) \cong \Gamma_{EG}(EG \times_G Y)$.

\section{Cell complexes and parametrized homotopy theory}\label{sec:cells}

The basic tension in parametrized homotopy theory is that homotopy groups are encoded in the fibers $E_b$ of an object $E$ parametrized over $B$, but a presentation of the homotopy theory in terms of cellular objects, either via model categories or classical CW methods, must involve total spaces.   In this section, we recall the approach to parametrized stable homotopy theory developed by May-Sigurdsson \cite{MS} in terms of $qf$-cells and contrast it with the usual Quillen notion of $q$-cells.  

The category $\Top/B$ of spaces over $B$ is the comma category of spaces 
\[
(E, p) = (p \colon E \arr B)
\]
equipped with a map $p$ to $B$.  As with any comma category, we may lift the Quillen-Serre $q$-model structure on $\Top$ to $\Top/B$.  The weak equivalences are the weak homotopy equivalences on total spaces.  The cofibrations, which we call \emph{$q$-cofibrations}, are the retracts of relative $\qI$-cell complexes, where
\[
\qI = \{ i(p) \colon (S^{n - 1}, p \circ i) \arr (D^{n}, p) \mid n \geq 0, \, p \colon D^n \arr B \}
\]
is the set of inclusions of the boundary of the $n$-disk, considered as a parametrized space via an arbitrary map $p$ to the base space $B$.  We use the term $q$-cell to mean a cell attached to a space over $B$ via pushout along a map in $\qI$.  The fibrations, which we call $q$-fibrations, are maps in $\Top/B$ that have the right lifting property against the set
\[
\qJ = \{j(p) \colon (D^{n}, p \circ j) \arr (D^n \times I, p) \mid  n \geq 1, \, p \colon D^n \times I \arr B \},
\]
or equivalently the maps that are Serre fibrations on total spaces.  In particular, the $q$-fibrant spaces over $B$ are precisely the Serre fibrations over $B$.  

The notion of a $q$-cofibration is in contrast to the notion of an $h$-cofibration, also known as an Hurewicz cofibration, which is a map $i \colon A \arr X$ having the homotopy extension property in the category $\Top$, meaning that every diagram of the form
\[
\xymatrix{
A \ar[d]_{i} \ar[r] & E^I \ar[d]^{\mathrm{ev}_0} \\
X \ar@{-->}[ur] \ar[r] & E
}
\]
admits a diagonal lift as indicated.  There is  a stronger notion of an $f$-cofibration, which is a map $i \colon A \arr X$ in the category $\Top/B$ having the homotopy extension property in the category $\Top/B$, meaning that a diagram as above of spaces over $B$ admits a lift that is a map of spaces over $B$.

There is also a pointed version of a parametrized space, which we call an ex-space.  The category $\Top_{B}$ of ex-spaces is the comma category of parametrized spaces 
\[
(E, p, s) = ( s \colon (B, \id_{B}) \arr (E, p) )
\]
over $B$ equipped with a section $s \colon B \arr E$. We let $F_B(E,E')$ refer to the space of maps from $E$ to $E'$ respecting both the inclusion of $B$ and the projection into $B$.

As with any comma category, we may lift the $q$-model structure on $\Top/B$ to $\Top_{B}$, the result of which we also call the $q$-model structure.  The weak equivalences and fibrations are detected by the forgetful functor to $\Top/B$.  The cofibrations are the retracts of relative $\qI_+$-cell complexes, where the set $\qI_{+}$ is obtained from $\qI$ by applying the left adjoint $(-)_+ \colon \Top/B \arr \Top_{B}$ of the forgetful functor, which freely adjoins a disjoint section.  We also use the term $q$-cell for a cell attached to an ex-space by pushout along a map in $\qI_{+}$.  The $q$-model structures on $\Top/B$ and $\Top_{B}$ are compactly generated with generating sets $(\qI, \qJ)$ and $(\qI_{+}, \qJ_{+})$, in the sense of \cite{MS}*{\S4.5}.  In particular, they are both cofibrantly generated model categories.

Recall that an orthogonal spectrum $X$ is an enriched diagram of based spaces $\{X_V: V \in \textup{ob}\mathscr J \}$, indexed by a category $\mathscr J$ that is also enriched in based spaces (see \cite[4.4]{MMSS} for its definition). 
\begin{definition}
We define a \emph{parametrized spectrum} $X$ to be a diagram of ex-spaces $X_V$ indexed by $\mathscr J$. This means that each morphism space $\mathscr J(V,W)$ is mapped continuously into the space $F_{B}(X_V,X_W)$ of maps of ex-spaces.  
%~ Of course, this definition requires us to specify a rule for tensoring the ex-spaces $X_V$ with the based spaces $\mathscr J(V,W)$. We take the external smash product $\mathscr J(V,W) \osma X_V$, which is defined as the pushout of $\mathscr J(V,W) \times X_V$ and $* \times B$ along $\mathscr J(V,W) \times B \cup * \times X_V$. The external smash product has the effect of smashing each fiber $(X_V)_b$ with the based space $\mathscr J(V,W)$.
For each $b \in B$, the diagram of fibers $(X_V)_b$ forms an ordinary orthogonal spectrum, which we call the \emph{fiber spectrum} and denote by $X_b$.
\end{definition}

Let $\Sp_B$ be the category of parametrized spectra over $B$. We define a \emph{level equivalence} to be a map of parametrized spectra $X \arr Y$ that gives an equivalence on each level, i.e. $X_V \arr Y_V$ is an equivalence of ex-spaces for each object $V$ in $\mathscr J$.

Next we define the stable equivalences. Suppose that $R(-)$ is a levelwise $q$-fibrant replacement functor. This means that it associates to each parametrized spectrum $X$ a parametrized spectrum $RX$ whose levels $(RX)_{V} \arr B$ are $q$-fibrations, and there is a natural level equivalence $X \arr RX$. Then we say that $X \arr Y$ is a \emph{stable equivalence} if the map of fibers $(RX)_b \arr (RY)_b$ is a $\pi_*$-isomorphism of orthogonal spectra, for every $b \in B$.  We will soon see that stable equivalences may be computed in terms of more general fibrant replacement functors $R$ (see Lemma \ref{lem:different_fib_replace}, and compare \cite{MS}*{12.3.4}). It is easy to check that the functor $\Sigma^{\infty}_{+} \colon \Top/B \arr \Sp_{B}$, which adds a disjoint section and then takes the fiberwise suspension spectrum, takes each weak homotopy equivalence of total spaces to a stable equivalence of parametrized spectra.

%~ Let $F_{V} \colon \Top_{B} \arr \Sp_{B}$ denote the shift-desuspension functor indexed by a real inner-product space $V$. By definition, $F_{V}$ is left adjoint to the functor that evaluates a $\mathscr J$-diagram at the object $V$. Then there is a model structure on $\Sp_B$ with the weak equivalences the level equivalences. The generating cofibrations and generating acyclic cofibrations are the maps of the form $F_{V} i(p)_+$ and $F_{V} j(p)_+$.

It is straightforward to build a model structure on $\Sp_B$ around the level equivalences. To be specific, let $F_{V} \colon \Top_{B} \arr \Sp_{B}$ denote the shift-desuspension functor indexed by a real inner-product space $V$. By definition, $F_{V}$ is left adjoint to the functor that evaluates a $\mathscr J$-diagram at the object $V$. Then there is a level $q$-model structure on $\Sp_B$, with the weak equivalences the level equivalences. The generating cofibrations and generating acyclic cofibrations are the maps of the form $F_{V} i(p)_+$ for $i(p)_+ \in \qI_{+}$ and $F_{V} j(p)_+$ for $j(p)_+ \in \qJ_{+}$.

However, when expanding the class of weak equivalences to the stable equivalences, one encounters a technical difficulty. The maps $i(p)$ and $i(p)_+$ are not $f$-cofibrations, meaning maps with the homotopy extension property in the category $\Top/B$.  They are only $h$-cofibrations, i.e. Hurewicz cofibrations.  As a consequence, applications of the gluing lemma that would allow inductive arguments over $q$-cell complexes fail in the fiberwise setting.  In particular, this prohibits the usual method of verifying that relative cell complexes built out of the generating acyclic $q$-cofibrations in $\Sp_{B}$ are weak equivalences.  There is no known ``$q$-''style model structure on $\Sp_{B}$ whose weak equivalences are the stable equivalences.

In order to construct a stable model structure on parametrized spectra, May-Sigurdsson develop an alternative model structure on parametrized spaces by replacing the $q$-cells with a more restrictive notion of $qf$-cells \cite{MS}*{\S6}.  Let
\[
\qfI = \{ i(p) \in \qI \mid \text{$i$ is an $f$-cofibration} \}
\]
be the set consisting of those generating $q$-cofibrations for which the inclusion $i \colon S^{n - 1} \arr D^n$ is an $f$-cofibration.  We call such a parametrized disk $p \colon D^n \arr B$ an $f$-disk over $B$, and note that $D^n$ has a collar neighborhood of the boundary whose map into $B$ lies entirely in the image of the boundary $S^{n-1}$.  For example, fix $b \in B$ and write $(D^n)^{b} = (D^n \oarr{r} \ast \oarr{b} B)$ for the parametrized space consisting of a single copy of $D^n$ as the fiber over $b$.  Then the inclusion
\[
i(b \circ r) \colon (S^{n - 1})^{b} \arr (D^n)^{b}
\]
is an element of $\qfI$.

By imposing a similar condition on $\qJ$ we get a subset $\qfJ \subset \qJ$, and we write $\qfI_{+}$ and $\qfJ_{+}$ for the sets of maps obtained by applying the disjoint section functor to $\qfI$ and $\qfJ$.

There is a compactly generated model structure on $\Top/B$ (respectively, $\Top_{B}$), called the $qf$-model structure, in which
\begin{itemize}
\item the weak equivalences are the weak homotopy equivalences of total spaces,
\item the cofibrations are the $qf$-cofibrations, which are the retracts of relative $\qfI$-cell complexes (respectively $\qfI_{+}$-cell complexes), and
\item the fibrations are the $qf$-fibrations; these are the maps which have the right lifting property with respect to $\qfJ$ (respectively $\qfJ_{+}$).
\end{itemize}
Notice that every $qf$-cofibration is a $q$-cofibration and that every $q$-fibration is a $qf$-fibration.  A quasifibration is a map $X \arr B$ for which the inclusion of each fiber into the corresponding homotopy fiber is a weak homotopy equivalence.  It is a crucial fact that every $qf$-fibration is a quasifibration \cite[6.5.1]{MS}, so we can still make arguments using the five-lemma as we normally would in fibration theory. 

Moving to the level of spectra, we first define a new set of generating $qf$-cofibrations:
\begin{align*}
\qfI_{\Sp} &= \{ F_{V} i(p)_+ \mid i(p) \in \qfI \}
\end{align*}
%~ We write $f^* \colon \ho \Sp_{B} \arr \ho \Sp_{A}$ for the derived base-change functor defined by pullback along a map of spaces $f \colon A \arr B$.
  %The functor $f^*$ has a left adjoint $f_{!}$ and a right adjoint $f_*$.
Then there is a compactly generated model structure on $\Sp_{B}$, called the stable model structure, in which the weak equivalences, cofibrations, and fibrations are:
\begin{itemize}
\item the stable equivalences as described above,
\item the $qf$-cofibrations, which are retracts of relative $\qfI_{\Sp}$-cell complexes, and
\item the $qf$-fibrations, which have the right lifting property with respect to all $qf$-cofibrations that are stable equivalences.
\end{itemize}
As this model structure is cofibrantly generated, it has generating acyclic cofibrations as well, but we will not need their precise definition. 
We use the term $qf$-cell to mean any cell attached to a parametrized space, ex-space, or parametrized spectrum by pushout along a map in $\qfI$, $\qfI_{+}$, or $\qfI_{\Sp}$.

%~ In our description of the stable equivalences of parametrized spectra, there is a potential for circular reasoning in that reference to derived functors presupposes a notion of weak equivalence in the domain category.  The actual definition of a stable equivalence is given by first approximating by a map of parametrized spectra which are level-wise $qf$-fibrations over $B$, then detecting the weak equivalence on all of the point-set fiber spectra.  We have given an equivalent formulation which is valid as soon as the derived base-change functors are well-defined, and is more conceptually useful. A consequence of the definition is that the functor $\Sigma^{\infty}_{+} \colon \Top/B \arr \Sp_{B}$, which adds a disjoint section and then takes the fiberwise suspension spectrum, takes weak homotopy equivalences of total spaces to stable equivalences of parametrized spectra.  Another consequence is that a map $f \colon M \arr N$ of parametrized spectra in which the domain and codomain are level-wise quasifibrations over $B$ is a stable equivalence if it induces a stable equivalence $f_{b} \colon M_b \arr N_b$ on each point-set fiber spectrum.  In particular, since the fibers of a quasifibration over two points in the same path component of the base are weak homotopy equivalent, it suffices to check this condition only for a representative point $b \in B$ in each path component of $B$.

We may use $qf$-fibrations, as well as the weaker notion of ex-quasifibration \cite{MS}*{\S8.5}, to detect stable equivalences of parametrized spectra.

\begin{lemma}\label{lem:different_fib_replace}\cite{MS}*{12.3.4, 12.4.1}
A map $X \arr Y$ of parametrized spectra is a stable equivalence if and only if every map of fiber spectra $RX_{b} \arr RY_{b}$ is a $\pi_*$-isomorphism of orthogonal spectra, where $R$ is a level-wise $qf$-fibrant approximation functor.  The same statement holds if $R$ is a level-wise ex-quasifibrant replacement functor .
\end{lemma}

We remark that the fibrant spectra $X$ in the stable model structure have the property that every level $X_V$ is a $qf$-fibrant space and therefore a quasifibration. It follows that a map of fibrant spectra $X \arr Y$ is a stable equivalence precisely when each map of fibers $X_b \arr Y_b$ is an equivalence of orthogonal spectra. Since the fibers of a quasifibration over two points in the same path component of the base are weak homotopy equivalent, it suffices to check this condition only for a representative point $b \in B$ in each path component of $B$.

\section{$\cC$-spectra and spectra over $B\cC$}\label{sec:parametrized_intro}

As discussed in the introduction, Theorem \ref{thm:intro_quillen_equiv_G} concerns spectra with an action of a topological group $G$, but we will begin our proof by working more generally with diagrams of spectra indexed by a topological category $\cC$.  While this scope of generality is not necessary to prove the main theorem, it allows for some interesting philosophical observations (see Remarks \ref{rem1} and \ref{rem2}) and presents no more difficulties than in working with a group.  The reader is encouraged to keep in mind throughout this section the special case where $\cC$ has a single object and morphism space $G$.

Let $\cC$ be any small topological category. We emphasize that we consider $\cC$ to be enriched in unbased spaces, not based spaces. For technical ease, we assume that every space $\cC(c, d)$ of morphisms is a retract of a CW complex.  In particular, the inclusion of the identity map $* \arr \cC(c,c)$ is an $h$-cofibration of unbased spaces. We write $\Sp$ for the category of orthogonal spectra \cite{MMSS}, equipped with the smash product $\sma$ and internal function spectra $F(-, -)$ that make it a closed symmetric monoidal category.  A $\cC$-diagram of spectra, or more briefly a $\cC$-spectrum, is a topological functor $X \colon \cC \arr \Sp$, where the topological enrichment of $\Sp$ is given by
\[
\Sp(X, Y) = \Omega^{\infty}F(X, Y).
\]
We can also view $\cC$ as a category enriched in spectra by taking the suspension spectra $\Sigma^{\infty} \cC(c, d)_+$ of the morphism spaces in $\cC$.  A $\cC$-spectrum may then be equivalently described as a spectral functor from this spectral category to $\Sp$. When $\cC$ is a topological monoid $G$, a $\cC$-spectrum is the same thing as a spectrum with an action of $G$, or equivalently a module spectrum over the ring spectrum $\Sigma^{\infty}_{+} G = \Sigma^{\infty} G_{+}$.

We write $X(c)_{V}$ for the $V$-th level of the spectrum $X(c)$ associated to an object $c$ of $\cC$.  We will take care to distinguish between \emph{levels} and \emph{objects}, so that a level-wise property of $\cC$-spectra is detected by the $\cC$-spaces $X(-)_{V}$ for all $V$, and an object-wise property is detected by the spectra $X(c)$ for all objects $c$ of $\cC$.  Let $\cC\Sp = \Fun(\cC, \Sp)$ denote the category of $\cC$-diagrams of spectra and natural transformations of functors.  We equip $\cC\Sp$ with the projective $q$-model structure, in which weak equivalences and fibrations are detected object-wise in the stable model structure on spectra \cite{MMSS}*{\S9}. We recall that:
\begin{itemize}
\item a morphism $f \colon X \arr Y$ of $\cC$-spectra is a weak equivalence (respectively, fibration) if for each object $c \in \cC$, the map of spectra $f(c) \colon X(c) \arr Y(c)$ is a stable equivalence (respectively $q$-fibration of spectra), and
\item the cofibrations are retracts of relative cell complexes built by attaching cells of the form
\[
 F_{V} (i \colon S^{n - 1}_+ \arr D^n_+) \sma \cC(c, -)_+ 
\]
where $F_V$ denotes the shift desuspension functor from based spaces to orthogonal spectra (left adjoint to evaluation at level $V$).
%\[
%I = \{ \cC(a, -)_+ \sma i \mid a \in \cC, i \in I_{\Sp}\} \quad J = \{ \cC(a, -)_+ \sma j \mid a \in \cC, j \in J_{\Sp}\},
%\]
%where $I_{\Sp}$ and $J_{\Sp}$ are the generating sets for the stable model structure on orthogonal spectra \cite{MMSS}*{\S9}.  
\end{itemize}
We use the terms \emph{$q$-cofibrations} and \emph{$q$-fibrations} for the cofibrations and fibrations in the projective $q$-model structure.  The definition is arranged so that evaluating at an object $c$ of $\cC$ and at a level $V$ preserves $q$-fibrations.  
%It is worth recalling that our model structures are all built on the Quillen model structure on the category of topological spaces---the prefix $q$- is a mnemonic for ``Quillen''.  In particular, a $q$-fibration of spaces is another name for a Serre fibration, and $q$-cofibrations are retracts of relative CW complexes.  The prefix also serves to distinguish from $h$-fibrations, i.e. Hurewicz fibrations, and $h$-cofibrations, which are defined in terms of the homotopy lifting and extension properties.

%If $X$ is a diagram for which each spectrum $X(a)$ is cofibrant, and the mapping spaces of $\cC$ are retracts of cell complexes, then one can use the bar construction $B(X,\cC,\cC)$ to replace $X$ by a weakly equivalent cofibrant diagram.

We will also use the category $\cC \Sp_{B} = \Fun(\cC, \Sp_{B})$ of $\cC$-diagrams of parametrized spectra over $B$.  We equip $\cC \Sp_{B}$ with the projective stable model structure, wherein weak equivalences and fibrations are detected in the stable model structure on $\Sp_{B}$ by the evaluation functors $M \longmapsto M(c)$ for each $c \in \ob \cC$.

The category $\Sp_B$ of parametrized spectra is enriched in $\Sp$ via the function spectrum objects $F_{B}(M, N)$ and is tensored over $\Sp$ via the half-parametrized smash product $M \osma X$ \cite{MS}*{11.4.10}, so that there is an adjunction
\[
\Sp_{B}( M \osma X, N) \cong \Sp(X, F_{B}(M, N)).
\]

Now suppose that $X$ is a (non-parametrized) $\cC$-spectrum and $M$ is a $\cC^{\op}$-diagram in $\Sp_{B}$.  The tensor product of functors $M \osma_{\cC} X$ is defined to be the coend of the functor \[
M(-) \osma X(-) \colon \cC \times \cC^{\op} \arr \Sp_{B}.
\]
In other words, $M \osma_{\cC} X$ is the coequalizer of the diagram of parametrized spectra
\[
\xymatrix{
\underset{c, d \in \ob \cC}{\bigvee} M(d) \osma \cC(c, d)_+ \sma X(c) \ar@<.5ex>[r]  \ar@<-.5ex>[r] & \underset{d \in \ob \cC}{\bigvee} M(d) \osma X(d)
}
\]
defined by the covariant and contravariant functorialities of $X$ and $M$.  %Here we implicitly index the coproducts over the objects of a skeleton of $\cC$ so that we can realize the coequalizer in our current universe.

\begin{lemma}\label{lem:pushout_product}
Let $i \colon M \arr N$ be a map of $\cC^{\op}$-spectra over $B\cC$, and let $j \colon X \arr Y$ be a map of $\cC$-spectra.  If $i$ is a $qf$-cofibration and $j$ is a $q$-cofibration, then the pushout product 
\[
i \boxempty j \colon (N \osma_{\cC} X) \cup_{M \osma_{\cC} X} (M \osma_{\cC} Y) \arr N \osma_{\cC} Y
\]
is a $qf$-cofibration of spectra over $B\cC$, and if either $i$ or $j$ is acyclic, then $i \boxempty j$ is as well.
\end{lemma}

\begin{proof}
Since the coend $\osma_{\cC}$ preserves colimits in each variable separately, so does the pushout-product.  Thus it suffices to prove the claim for the generating cofibrations and generating acyclic cofibrations. Assume that $M \arr N$ is of the form $i \osma \cC(-,d)_+$ where $(i \colon A \arr B) \in \qfI_{\Sp}$ is a generating $qf$-cofibration in $\Sp_{B\cC}$, and $X \arr Y$ is of the form $j \sma \cC(c,-)_+ $ where $j \colon C \arr D$ is a generating $q$-cofibration of orthogonal spectra. The pushout-product then becomes
\[
i \boxempty j \colon [(B \osma C) \cup_{A \osma C} (A \osma D) \arr A \osma D] \sma \cC(c,d)_+ 
\]
and since the spaces $\cC(c, d)$ are retracts of CW complexes this reduces to the same claim for the external smash product of a spectrum over $B\cC$ with an ordinary spectrum, found in \cite{MS}*{12.6.5}.
\end{proof}

Let $E\cC$ denote the $\cC^{\op}$-diagram of spaces whose value on an object $c$ of $\cC$ is the two-sided bar construction 
\[
E\cC^{c} = B(\ast, \cC, \cC(c, -)) = \Bigl \vert [q] \longmapsto \coprod_{c_0, \dotsc, c_q \in \ob \cC} \cC(c_{q - 1}, c_q) \times \dotsm \times \cC(c_0, c_1) \times \cC(c, c_0) \Bigr\vert.
\]
Each of the spaces $E\cC^c$ is contractible, by the usual contracting homotopy argument. The collapse map $\cC(c, -) \arr \ast$ induces a map $\pi \colon E\cC^{c} \arr B \cC$, so we may consider $E\cC$ as a $\cC^{\op}$-diagram in spaces over $B\cC$.  Of course, in the special case of a topological group $G$, the map $\pi$ is the usual projection $EG \arr BG$ with a fiberwise $G$-action.

\begin{lemma}
$E\cC$ is a $q$-cofibrant diagram in the projective model structure on $\cC^{\op}\Top/B\cC$.
\end{lemma}

\begin{proof}
We recall that the projective model structure for diagrams of unbased spaces $\cC^{\op}\Top$ has a similar definition to the one for spectra given above, except that all the products are Cartesian products of unbased spaces instead of smash products of based spaces. It suffices to show that the latching maps for the above simplicial space are $q$-cofibrations when regarded as maps in $\cC^{\op}\Top$. The action of $\cC$ preserves the coproduct over the tuples of objects $c_0, \dotsc, c_q \in \ob \cC$, so we may focus on one such tuple. Arguing as in \cite[4.8]{malk_cyclotomic}, the latching map is an iterated pushout-product of the maps $\emptyset \arr \cC(c_{i-1},c_i)$ and ${*} \arr \cC(c_i,c_i)$ (which are $q$-cofibrations by our assumption on $\cC$), times the free $\cC^{\op}$-diagram $\cC(-,c_0)$. This gives a product of a $q$-cofibration of unbased spaces and a free diagram, which is a $q$-cofibration of diagrams of unbased spaces.
\end{proof}

We warn the reader that in general $E\cC$ is \emph{not} a $qf$-cell complex, and may not be $qf$-cofibrant. This can be checked by hand for the double cover $E \bZ/2 \arr B\bZ/2$. We also warn that the projection $\pi$ is often not a fibration, or even a quasifibration, when $\cC$ is not a topological group.

By adding a disjoint section and taking the fiberwise suspension spectrum, we obtain a cofibrant $\cC^{\op}$-diagram of parametrized spectra $\Sigma^{\infty}_{B\cC} E\cC_+$, which we allow ourselves to abbreviate to $E\cC_+$.  We write $E'\cC_+ = \Sigma^{\infty}_{B\cC} B(\ast, \cC, \cC)^{qf\mathrm{-cof}}_+$ for a $qf$-cofibrant approximation of it in the projective stable model structure on $\cC^{\op} \Sp_{B\cC}$.  The cofibrant approximation may be calculated by first taking a $qf$-cofibrant approximation of $E\cC$ in $\cC^{\op}\Top/B \cC$, then applying the disjoint section and fiberwise suspension spectrum functor. 

Let $\beta \colon \cC\Sp \arr \Sp_{B\cC}$ denote the functor which takes a $\cC$-diagram $X$ to the parametrized spectrum 
\[
\beta(X) = %\Sigma^{\infty}_{B\cC} QB(\ast, \cC, \cC)_+ 
E'\cC_+ \osma_{\cC} X,
\]
and similarly on morphisms. Note that we have implicitly taken the fiberwise suspension spectrum of $E'\cC_+$.  Alternatively, we use the space-level variant of $\osma$ to take the coend of the ex-space $E'\cC_+$ and each spectrum level of $X$, and this gives the same result.

%Let $L \colon \cC\Sp \arr \Sp_{B\cC}$ denote the functor which on objects takes a diagram $X \colon \cC \arr \Sp$ to the parametrized spectrum defined at level $V$ by the ex-space
%\[
%\xymatrix{
%LX_{V} = 
%B(*, \cC, X(-)_{V}) \ar@<.5ex>[r] & B(*, \cC, *) \ar@<.5ex>[l]
%}
%\]
%whose projection to $B\cC$ is defined by collapsing $X(-)_{V}$ to a point, and whose section is determined by the map $* \arr X(-)_{V}$ assigning each space in the diagram its basepoint.  It is important here that we use the bar construction valued in unbased spaces, so that the basepoints assemble to a section, and not a single point in $LX_{V}$.  The structure maps 
%\[
%B(*, \cC, X(-)_{V}) \sma_{B\cC} S^{W}_{B\cC} \arr B(*, \cC, X(-)_{V \oplus W})
%\]
%of the parametrized spectrum $LX$ are induced by the natural transformation $X(-)_{V} \sma S^W \arr X(-)_{V \oplus W}$ defined by the structure maps of the component spectra $X(c)$.  

We write $E'\cC^{c}_{+}$ for the parametrized spectrum obtained by evaluating the $\cC^{\op}$-diagram $E'\cC_+$ at $c$.
The functor $\beta$ has a right adjoint $\Phi$ which takes a parametrized spectrum $M$ over $B\cC$ to the $\cC$-spectrum
\[
c \longmapsto F_{B\cC}( %\Sigma^{\infty}_{B\cC} QB(\ast, \cC, \cC(c, -))_+
E'\cC^{c}_+, M).
\]
The notation for the adjunction $(\beta, \Phi)$ is meant to suggest \emph{Borel} construction and homotopy \emph{fiber}.  The following is a consequence of Lemma \ref{lem:pushout_product}.

\begin{theorem}
The adjunction $(\beta, \Phi) \colon \cC\Sp \rightleftarrows \Sp_{B\cC}$ is a Quillen adjunction with respect to the projective $q$-model structure on $\cC\Sp$ and the stable model structure on $\Sp_{B\cC}$.  
\end{theorem}

In \S\ref{sec:groupoid}, we will prove that $(\beta, \Phi)$ is a Quillen equivalence when $\cC$ is a groupoid up to homotopy.  In preparation for our later work, we now give a precise meaning to the idea that $\Phi$ plays the role of the homotopy fiber. 

The identity morphism of an object $c$ of $\cC$ determines a basepoint $* \in E\cC^{c}$, which is non-degenerate by our hypotheses on $\cC$.  Choose a lift of the basepoint along the $qf$-cofibrant approximation map $\alpha \colon E'\cC^{c} \arr E\cC^{c}$.  
Given a parametrized spectrum $M$ over $B\cC$, restriction to this basepoint defines a map 
\[
\rho \colon F_{B\cC} (E'\cC^{c}_+, M) \arr M_{[c]}
\]
to the fiber of $M$ over the point $[c] \in B\cC$ determined by the object $c$.

\begin{lemma}\label{lemm:rho_equiv}
If $M$ is a $qf$-fibrant spectrum over $B\cC$, then the restriction map $\rho$ is a stable equivalence of spectra.
\end{lemma}

In fact, we will prove that it is a level equivalence of spectra. Evaluating at an indexing space $V$, the domain is given by the space $\Map_{B\cC}(E'\cC, M(V))$ of maps of spaces over $B\cC$.  Thus it suffices to prove the corresponding result for parametrized spaces, which is the second part of the next lemma.

\begin{lemma}\hspace{2in}\label{lemm:rho_equiv_space_level}
\begin{itemize}
\item[(i)] If $Y \arr B\cC$ is an $h$-fibration, then the map
\[
\tilde\rho \colon \Map_{B\cC}(E\cC^{c}, Y) \arr Y_{[c]}
\]
induced by restriction to the basepoint of $E\cC^{c}$ is a weak homotopy equivalence of spaces.
\item[(ii)] If $Y \arr B\cC$ is a $qf$-fibration, then the restriction map
\[
\rho \colon \Map_{B\cC}(E'\cC^{c}, Y) \arr Y_{[c]}
\]
is a weak homotopy equivalence of spaces.
\end{itemize}
\end{lemma}

\begin{proof}
For part (i) we follow Shulman's proof of \cite[Lem. 8.3]{Sh08}.  Here we use the fact that the inclusion of the basepoint $\ast \arr E\cC^{c}$ is an $h$-cofibration.  Since the codomain is contractible, the inclusion is also a homotopy equivalence, and thus, together with the $h$-fibration $Y \arr B\cC$, it induces an acyclic $h$-fibration
\[
\Map(E\cC^{c}, Y) \arr \Map(*, Y) \times_{\Map(*, B\cC)} \Map(E\cC^{c}, B\cC) \cong Y \times_{B\cC} \Map(E\cC^{c}, B\cC)
\]
on mapping spaces.  Pulling back along the inclusion $y \mapsto (y, \pi)$ of the fiber $Y_{[c]}$ into $Y \times_{B\cC} \Map(E\cC^{c}, B\cC)$ gives the restriction map $\tilde\rho$, which is therefore an acyclic $h$-fibration.

We prove part (ii) by reducing to part (i). Let $Y \arr Y'$ be an approximation of $Y$ by an $h$-fibration over $B\cC$, say by the usual mapping space construction.  The approximation map induces the upwards pointing maps in the next commutative diagram
\[
\xymatrix{
\Map_{B\cC}(E\cC^{c}, Y') \ar[d]_{\simeq} \ar[dr]^-{\tilde\rho} & \\
\Map_{B\cC}(E'\cC^{c}, Y') \ar[r]_-{\rho} & Y'_{[c]} \\
\Map_{B\cC}(E'\cC^{c}, Y) \ar[r]_-{\rho} \ar[u]^{\simeq} & Y_{[c]} \ar[u]_{\simeq},
}
\]
and they are both weak homotopy equivalences since $F_{B\cC}(E'\cC^{c}, -)$ preserves weak equivalences between $qf$-fibrant parametrized spaces.  The downward pointing arrow is induced by the $qf$-cofibrant approximation map $E'\cC^{c} \arr E\cC^{c}$, and it is a weak homotopy equivalence because $Y'$, being an $h$-fibration over $B\cC$, is also a $q$-fibration, and $E\cC^{c}$ and $E'\cC^{c}$ are both $q$-cofibrant. The diagonal restriction map $\tilde\rho$ is a weak homotopy equivalence by part (i), and this concludes the proof.
\end{proof}

\section{The aggregate model structure}\label{sec:aggregate}

In this section we introduce the aggregate model structure on the category of $\cC$-spectra and prove that it is Quillen equivalent to May-Sigurdsson's stable model structure on parametrized spectra over $B\cC$. In order to motivate our definition of aggregate equivalences, we begin with a lemma. Let $X$ be any $\cC$-spectrum, not necessarily cofibrant.
\begin{lemma}\label{lem:is_actually_hocolim}
The parametrized spectrum $E\cC_+ \osma_{\cC} X$ is at level $V$ homeomorphic to the Bousfield-Kan model for the unbased homotopy colimit $\hocolim_{\cC} X(-)_{V}$. The inclusion and projection onto $B\cC$ are identified with the inclusion and projection onto the diagram consisting of just the basepoints of the spaces $X(c)_{V}$.
\end{lemma}

\begin{proof}
This is expected in the case of a group $G$, since $EG_+ \osma_G X$ is on each fiber over $BG$ homeomorphic to $G_+ \sma_G X_b$, which is homeomorphic to $X_b$. For the general case, we first check using the defining adjunction for the half-parametrized smash product of spaces
\[
F_B( M \osma X, N) \cong F(X, F_B(M, N)).
\]
that for a space $E$ over $B$ and a based space $X$, we have a canonical homeomorphism
\[ E_+ \osma X \cong (E \times X) \cup_{(E \times *)} (B \times *). \]
Applying this to the coequalizer system for $E\cC_+ \osma_{\cC} X_V$
%\[
%\xymatrix{
%\underset{c, d \in \ob \cC}{\bigvee} (E\cC(d) \times \cC(c, d))_+ \osma X(c)_V \ar@<.5ex>[r]  \ar@<-.5ex>[r] & \underset{d \in \ob \cC}{\bigvee} E\cC(d)_+ \osma X(d)_V
%}
%\]
gives
\[
\xymatrix{
\underset{c, d \in \ob \cC}{\bigvee} (E\cC(d) \times \cC(c, d) \times X(c)_V) \cup_{E\cC(d) \times \cC(c, d) \times *} (B\cC \times *) \ar@<.5ex>[d]  \ar@<-.5ex>[d] \\ \underset{d \in \ob \cC}{\bigvee} (E\cC(d) \times X(d)_V) \cup_{E\cC(d) \times *} (B\cC \times *).
}
\]
As the big wedge signifies coproduct over $B\cC$, this rearranges to
\[
\xymatrix{
\left(\underset{c, d \in \ob \cC}{\coprod} E\cC(d) \times \cC(c, d) \times X(c)_V)\right) \cup_{\coprod_{c,d} E\cC(d) \times \cC(c,d)} B\cC \ar@<.5ex>[d]  \ar@<-.5ex>[d] \\ \left(\underset{d \in \ob \cC}{\coprod} E\cC(d) \times X(d)_V \right) \cup_{\coprod_d E\cC(d)} B\cC
}
\]
Since coequalizers commute with pushouts, this rearranges to the pushout
\[ E\cC \times_\cC X(-)_V \cup_{E\cC \times_\cC {*}} B\cC \]
which simplifies to the coend of unbased spaces
\[ E\cC \times_\cC X(-)_V \cong \Bigl \vert [q] \longmapsto \coprod_{c_0, \dotsc, c_q \in \ob \cC} \cC(c_{q - 1}, c_q) \times \dotsm \times \cC(c_0, c_1) \times X(c_0)_V \Bigr\vert. \]
This is the Bousfield-Kan model for the unbased homotopy colimit of $X(-)_V$.
\end{proof}

We define a new notion of equivalence of $\cC$-spectra by ``aggregating'' each diagram together into a single parametrized spectrum, then measuring its homotopy type.

\begin{definition}  
A map $f \colon X \arr Y$ of $\cC$-spectra is an \emph{aggregate equivalence} if the induced map 
\[
E\cC_{+} \osma_{\cC} X \arr E\cC_{+} \osma_{\cC} Y
\]
is a stable equivalence of parametrized spectra over $B\cC$.
\end{definition}

In other words, the aggregate equivalences are detected by taking (unbased) homotopy colimits over $\cC$. We remark that $X$ and $Y$ are not assumed to be cofibrant diagrams in this definition. This will not get us into trouble because this particular model of the homotopy colimit is known to preserve all equivalences of diagrams of spaces. In particular, every level equivalence of $\cC$-spectra is an aggregate equivalence.

In order to verify that object-wise stable equivalences of $\cC$-spectra are also aggregate equivalences, we may without loss of generality make $X$ and $Y$ cofibrant diagrams of spectra. The following lemma, combined with Lemma \ref{lem:pushout_product} finishes the verification.
\begin{lemma}\label{lemm:cofibrant_detection}
%\begin{itemize}
%\item[(i)] 
If $X$ is a $q$-cofibrant $\cC$-spectrum, then the map induced by the cofibrant approximation of $E\cC$
\[
\alpha \osma \id_{X} \colon E' \cC_+ \osma_{\cC} X \arr E\cC_{+} \osma_{\cC} X
\]
is a level-wise equivalence of parametrized spectra over $B\cC$.  The same statement holds when $X$ is a $q$-cofibrant based $\cC$-space.
%\item[(ii)] If $X$ is a $q$-cofibrant based $\cC$-space, then the map induced by the cofibrant approximation of $E\cC$
%\[
%\alpha \osma \id_{X} \colon E' \cC_+ \osma_{\cC} X \arr E\cC_{+} \osma_{\cC} X
%\]
%is a weak homotopy equivalence of spaces.
%\end{itemize}
\end{lemma}
\begin{proof}
The first claim is an immediate consequence of the second. For the second, we argue just as in the proof of Lemma \ref{lem:pushout_product} that the coend
\[
- \osma_{\cC} - \colon \cC^{\op}\Top_{B\cC} \times \cC\Top_* \arr \Top_{B\cC}
\]
is a Quillen bifunctor for the projective $q$-model structures on $\cC^{\op}\Top_{B\cC}$ and $\cC\Top_*$, and the $q$-model structure on $\Top_{B\cC}$.
%~ It is crucial here that we use the $q$-model structure and not the $qf$-model structure on the categories of ex-spaces.
Since $E\cC$ and $E'\cC$ are both projective $q$-cofibrant, and $X$ is assumed to be projective $q$-cofibrant, the weak homotopy equivalence $\alpha$ is preserved by by taking the coend with $X$.
\end{proof}

Now that we know that the aggregate equivalences contain the object-wise stable equivalences, we may construct the left Bousfield localization of the projective $q$-model structure on $\cC$-spectra along the aggregate equivalences.  Since the aggregate equivalences are detected by taking homotopy colimits, the resulting model structure is a stable analog of the model structure on certain categories of diagram spaces \citelist{\cite{diagram_spaces} \cite{SaSch}}.  When the category $\cC$ is discrete, the aggregate model structure was considered by Cisinski \cite{cisinski}*{\S2}.  The existence of the Bousfield localization exists by Hirschhorn's general criteria \cite{Hirschhorn}*{4.1.1}, and can be constructed by hand using the method axiomatized in \cite{MS}*{\S4.5}.  

\begin{theorem}\label{thm:agg_model_str}
There is a compactly generated model category structure on $\cC\Sp$, which we call the \emph{aggregate model structure}, whose weak equivalences are the aggregate equivalences and whose cofibrations coincide with the $q$-cofibrations in the projective model structure.  The fibrations are the object-wise $q$-fibrations $X \arr Y$ for which every diagram
\[ \xymatrix{
X(c) \ar[r] \ar[d] & X(d) \ar[d] \\
Y(c) \ar[r] & Y(d) } \]
induced by a morphism $c \arr d$ in $\cC$ has the property that the map from the initial vertex to the pullback is a weak homotopy equivalence. In particular, the fibrant diagrams are \emph{almost constant}, meaning that every map $X(c) \arr X(d)$ induced by a morphism in $\cC$ is a stable equivalence of spectra.
\end{theorem}

By Lemma \ref{lemm:cofibrant_detection}, the Borel construction $\beta \colon \cC\Sp \arr \Sp_{B\cC}$ takes aggregate equivalences between cofibrant diagrams to stable equivalences of parametrized spectra.  Thus the adjunction $(\beta, \Phi)$, which we already know to be Quillen for the projective model structure, is also Quillen for the aggregate model structure.  The main theorem of this section is a strengthening of this observation:

\begin{theorem}\label{thm:agg_quillen_equiv}
The adjunction $(\beta, \Phi)$ is a Quillen equivalence between the aggregate model structure on $\cC$-spectra and the stable model structure on parametrized spectra over $B\cC$.
\end{theorem}

In order to prove the theorem, we will use an elementary fact from category theory: if $(F, G)$ is an adjunction of categories, the left adjoint $F$ detects isomorphisms, and the counit $\epsilon \colon FGx \arr x$ is an isomorphism on every object $x$, then $(F, G)$ is an equivalence of categories.  In our case, we know that the adjunction $(\beta, \Phi)$ is Quillen, so we may apply the fact to the induced adjunction $(\bL \beta, \bR \Phi)$ of homotopy categories.  Note that the left adjoint $\beta$ detects weak equivalences between $q$-cofibrant $\cC$-spectra by Lemma \ref{lemm:cofibrant_detection}, and so the left derived functor $\bL \beta$ detects isomorphisms.  This means that the proof of Theorem \ref{thm:agg_quillen_equiv} is reduced to the proof of:

\begin{proposition}\label{prop:two_claims}
For every $qf$-fibrant parametrized spectrum $M$ over $B\cC$, the derived counit
\[
E'\cC_+ \osma_{\cC} F_{B\cC}(E'\cC_+, M)^{qf\mathrm{-cof}} \arr E'\cC_+ \osma_{\cC} F_{B\cC}(E'\cC_+, M) \oarr{\epsilon} M
\]
obtained by precomposing $\epsilon$ with a cofibrant approximation of the $\cC$-spectrum $F_{B\cC}(E'\cC_+, M)$ is a level equivalence.
\end{proposition}

%~ The next result implies claim (i).

%~ \begin{lemma}\label{lemm:cofibrant_detection}
%~ If $X$ is a $q$-cofibrant $\cC$-spectrum, then the map induced by the cofibrant approximation of $E\cC$
%~ \[
%~ \alpha \osma \id_{X} \colon E' \cC_+ \osma_{\cC} X \arr E\cC_{+} \osma_{\cC} X
%~ \]
%~ is a level equivalence of parametrized spectra over $B\cC$.
%~ \end{lemma}

%~ To prove the lemma, it suffices to argue that at each level the map is a weak homotopy equivalence of total spaces, and so we are reduced to proving:

%~ \begin{lemma}\label{lemm:cofibrant_detection_spacelevel}
%~ If $X$ is a $q$-cofibrant based $\cC$-space, then the map induced by the cofibrant approximation of $E\cC$
%~ \[
%~ \alpha \osma \id_{X} \colon E' \cC_+ \osma_{\cC} X \arr E\cC_{+} \osma_{\cC} X
%~ \]
%~ is a weak homotopy equivalence of spaces.
%~ \end{lemma}
%~ \begin{proof}
%~ The coend
%~ \[
%~ - \osma_{\cC} - \colon \cC^{\op}\Top_{B\cC} \times \cC\Top_* \arr \Top_{B\cC}
%~ \]
%~ is a Quillen bifunctor for the projective $q$-model structures on $\cC^{\op}\Top_{B\cC}$ and $\cC\Top_*$, and the $q$-model structure on$\Top_{B\cC}$.  It is crucial here that we use the $q$-model structure and not the $qf$-model structure on the categories of ex-spaces.  Since $E\cC$ and $E'\cC$ are both projective $q$-cofibrant, and $X$ is assumed to be projective $q$-cofibrant, the weak homotopy equivalence $\alpha$ is preserved by by taking the coend with $X$.
%~ \end{proof}

To prove this we use the following lemma, which forms the heart of Quillen's proof of his ``Theorem B'' \cite{Quillen}*{p. 98}.  For a proof that applies to topological categories, see the very general account due to Meyer \cite{Meyer}*{4.4.1}

\begin{lemma}\label{lemm:thmB}
Let $X$ be a based $\cC$-space for which every map $X(c) \arr X(d)$ of spaces induced by a morphism in $\cC$ is a weak homotopy equivalence.  Then the projection map $\pi_X \colon E\cC_+ \osma_{\cC} X \arr B\cC$ is a quasifibration whose fiber over a point $[c] \in B\cC$ is naturally isomorphic to $X(c)$.
\end{lemma}

\begin{proof}[Proof of Proposition \ref{prop:two_claims}]
Since cofibrant approximations of diagrams of spectra are level-wise equivalences, it suffices to consider the case where $M$ is a $qf$-fibrant space. Without loss of generality we may replace it by an $h$-fibrant space. Consider the following commuting diagram, in which our desired composite is the bottom edge.
\[ \xymatrix @!C{
E\cC_+ \osma_\cC F_{B\cC}(E\cC_+, M)^{qf\mathrm{-cof}} \ar[r] & E\cC_+ \osma_\cC F_{B\cC}(E\cC_+, M) \ar[rd]^-\epsilon & \\
E'\cC_+ \osma_\cC F_{B\cC}(E\cC_+, M)^{qf\mathrm{-cof}} \ar[u]_-{\alpha \osma \id} \ar[d]^-{\id \osma F(\alpha,\id)} \ar[r] & E'\cC_+ \osma_\cC F_{B\cC}(E\cC_+, M) \ar[u]_-{\alpha \osma \id} \ar[d]^-{\id \osma F(\alpha,\id)} & M \\
E'\cC_+ \osma_{\cC} F_{B\cC}(E'\cC_+, M)^{qf\mathrm{-cof}} \ar[r] & E'\cC_+ \osma_{\cC} F_{B\cC}(E'\cC_+, M) \ar[ru]^-\epsilon & \\
}\]
It suffices to show that the maps along the top and left edges are equivalences of spaces. For the left edge this follows from the fact that $F(\alpha,\id)$ is an equivalence of spaces and Lemma \ref{lemm:cofibrant_detection}. For the horizontal map on the top edge this follows from Lemma \ref{lem:is_actually_hocolim}. For the remaining map $\epsilon$, by Lemma \ref{lemm:thmB} it suffices to check that the map of fibers $\epsilon_{[c]}$ is an equivalence. We check that its composition with the canonical isomorphism of Lemma \ref{lemm:thmB} gives the equivalence $\tilde\rho$ of Lemma \ref{lemm:rho_equiv_space_level}(i):
\[ \xymatrix{
F_{B\cC}(E\cC^c_+, M) \ar[r]^-\cong & (E\cC_+ \osma_\cC F_{B\cC}(E\cC_+, M))_{[c]} \ar[r]^-{\epsilon_{[c]}} & M_{[c]}
} \]
This finishes the proof.
\end{proof}

\begin{remark}
The analog of Theorem \ref{thm:agg_quillen_equiv} in the case of $\cC$-diagrams of simplicial sets is proved by Heuts and Moerdijk \cite{HM}.  Their route is different than ours; instead of directly proving that a localization of the projective model structure on $\cC$-diagrams is Quillen equivalent to the Kan-Quillen model structure on the comma category $\sSet/N\cC$ of simplicial sets over the nerve, they first construct a model structure on $\sSet/N\cC$ related to the Joyal model structure, prove that it is Quillen equivalent to the projective model structure, then deduce the former statement as a consequence.
\end{remark}

\section{The case where $\cC$ is a groupoid up to homotopy}\label{sec:groupoid}

In this section, we assume that the category $\pi_0 \cC$ obtained from the topological category $\cC$ by applying $\pi_0$ to each mapping space is a groupoid.  In other words, we assume that $\cC$ is a groupoid up to homotopy.  We will prove the following theorem and then derive some consequences.

\begin{theorem}\label{thm:aggregate_is_projective}
When $\pi_0 \cC$ is a groupoid, the aggregate model structure and the projective model structure on the category of $\cC$-spectra coincide.  Consequently, the adjunction $(\beta, \Phi)$ is a Quillen equivalence between the projective model structure on $\cC$-spectra and the stable model structure on parametrized spectra over $B\cC$.
\end{theorem}

To prove this, we first observe that Lemma \ref{lemm:thmB} and our assumption on $\cC$ imply the following lemma.
%, which is somewhat famous in fibration theory.

\begin{lemma}
Suppose that $\pi_0 \cC$ is a groupoid and that $X$ is a $\cC$-space.   Then the canonical projection $\hocolim_{\cC} X \arr B\cC$ is a quasifibration whose fiber over a point $[c] \in B\cC$ is naturally isomorphic to $X(c)$.
\end{lemma}

Then the proof of Theorem \ref{thm:aggregate_is_projective} follows from the lemma.  Indeed, it suffices to prove that the aggregate equivalences and the object-wise stable equivalences coincide.  By definition, a map $f \colon X \arr Y$ of $\cC$-spectra is an aggregate equivalence precisely when the induced map $E\cC_{+} \osma_{\cC} f$ is a stable equivalence of parametrized spectra over $B\cC$.  By the lemma, the domain and codomain of this map are level-wise quasifibrations over $B\cC$.  Thus, $f$ is an aggregate equivalence if and only if each map $(E\cC_{+} \osma_{\cC} f)_{[c]}$ of fiber spectra is a stable equivalence.  We may identify such a map with the map of spectra $f(c) \colon X(c) \arr Y(c)$ indexed by the object $c$ of $\cC$, so the proof is complete.

\begin{corollary}\label{thm:equiv_G_monoid}
Suppose that $G$ is a group-like topological monoid that is a retract of a CW complex.  Then $(\beta, \Phi)$ is a Quillen equivalence between the projective $q$-model structure on the category of spectra with an action of $G$ and the stable model structure on the category of parametrized spectra over $BG$.
\end{corollary}

%\begin{remark}
%Our approach is very similar to a standard approach to deriving the right Kan extension of a diagram using the projective model structure. The right Kan extension along $\cD \arr \cC$ is not a Quillen right adjoint, but it sits as part of a larger adjunction of two variables from $\cC$ times $\cD$ to $\cD$, where the tensoring is the coend of a diagram over $\cD$ and a restriction of a diagram over $\cC$. Deriving the maps from the free diagrams in $\cD$ to the given diagram in $\cD$ gives the homotopy right Kan extension; its left adjoint is tensoring the restriction of the $\cC$ diagram with free $\cD$ diagrams, which is essentially applying the left adjoint and then cofibrantly replacing afterwards.
%\end{remark}

This is just a more detailed statement of Theorem \ref{thm:intro_quillen_equiv_G} from the introduction, and has now been proved. Although this corollary was the motivation for the present paper, the generality of Theorem \ref{thm:aggregate_is_projective} provides us with many more model categories that are also Quillen equivalent to the category of parametrized spectra. We summarize two interesting and somewhat related special cases in the remarks below.

\begin{remark}\label{rem1}
Let $X$ be a CW complex and let $\cC$ be the category of simplices over $X$, so that $B\cC$ is weakly equivalent to $X$. The objects of $\cC$ are contractible spaces, so a fibrant $\cC$-diagram in the aggregate model structure may be interpreted as a homotopy functor. By homotopy left Kan extension and restriction, these correspond to the functors on the category of spaces over $X$ that are \emph{linear} in the sense of Goodwillie calculus. Our Quillen equivalence gives more structure to the statement by Goodwillie \cite{calc3}*{\S5} that linear functors on spaces over $X$ correspond to parametrized spectra over $X$. However, it stops short of giving a direct Quillen equivalence between the two. It would be interesting to see if our techniques may be combined with those of \citelist{\cite{BCR} \cite{BR}} to give such a direct equivalence.
\end{remark}

\begin{remark}\label{rem2}
Similar to the previous remark, let $\cC$ be a category consisting of contractible open subsets of $X$, which cover $X$ and are closed under finite intersections, so that the classifying space $B\cC^\op$ is equivalent to $X$. One can include $\cC$ into the larger category $\cD$ of all open subsets of $X$. In light of remarks such as \cite{de2012manifold}*{2.3}, it is reasonable to define a \emph{homotopy sheaf of spectra} to be any diagram equivalent to a homotopy right Kan extension of a diagram on $\cC^\op$ to $\cD^\op$.
% I'm stopping short of claiming this is a definition. I can check that things work out fine for a single cover, but to be a sheaf you have to play nice on *all* good covers.
Under this interpretation, Theorem \ref{thm:aggregate_is_projective} says that parametrized spectra over $X$ model locally constant homotopy sheaves of spectra over $X$, lifting an earlier result of Shulman for parametrized spaces \cite{Sh08}. Although we don't pursue in this paper the technical issues needed to make this correspondence precise, we consider it an interesting question whether our techniques could in fact be used to relate parametrized spectra to locally-constant sheaves living in existing model structures on presheaves of spectra on $\cD$, such as those found in \citelist{\cite{jardine} \cite{block_lazarev}}.
\end{remark}

\section{Indexed symmetric monoidal categories}\label{sec:indexed}

In this section, we check that the equivalence between the homotopy category of $G$-spectra and the homotopy category of spectra parametrized over $BG$ respects derived base-change functors on either side.  We use the language of indexed symmetric monoidal categories \cite{PS12} to describe this agreement of base-change functors.

Let $S$ be a cartesian monoidal category.  An $S$-indexed symmetric monoidal category is a pseudofunctor $\cM \colon S^{\op} \arr \mathrm{SymmMonCat}$ to the 2-category of symmetric monoidal categories, strong symmetric monoidal functors, and monoidal transformations.  In other words, for each object $A$ of $S$ there is a symmetric monoidal category $\cM^{A}$, and for each morphism $f \colon A \arr B$ there is a strong symmetric monoidal functor $f^* \colon \cM^{B} \arr \cM^{A}$, along with natural monoidal isomorphisms $(g \circ f)^* \cong f^* \circ g^*$ and $(\id_{A})^* \cong \id_{\cM^{A}}$ satisfying associativity and unit conditions.

We are interested in the case where the indexing category $S$ is the category $\Top$ of topological spaces, the category $\TopGrp$ of well-based group-like topological monoids, or the category $\Top_*^{\mathrm{conn}}$ of based connected topological spaces.
 
\begin{example}
The assignment $B \longmapsto \ho \Sp_{B}$, taking the space $B$ to the homotopy category of parametrized spectra over $B$, along with the derived base-change functors $f^* \colon \ho \Sp_{B} \arr \ho \Sp_{A}$, defines a $\Top$-indexed symmetric monoidal category $\ho \Sp_{(-)}$.
\end{example}
 
\begin{example}
Similarly, the assignment $G \longmapsto \ho \Sp_{BG}$ taking $G$ to the homotopy category of parametrized spectra over $BG$ defines a $\TopGrp$-indexed symmetric monoidal category $\ho \Sp_{B(-)}$. We assume that $G$ is well-based so that equivalences of groups $G \arr G'$ give equivalences of spaces $BG \arr BG'$ built using the bar construction.
%~ Strictly speaking, if we let $G$ vary over all grouplike monoids, then we ought to take a functorial replacement of $G$ by a well-based group with the homotopy type of a cell complex, before defining $BG$. However the indexed monoidal category this describes is equivalent to the more na\" ively defined one when it is restricted to this subcategory of nice topological groups. We may 
\end{example}

\begin{example}
The assignment $G \longmapsto \ho \Mod_{\Sigma^\infty_+ G}$ taking $G$ to the homotopy category of module spectra over the ring spectrum $\Sigma^{\infty}_{+} G$ defines a $\TopGrp$-indexed symmetric monoidal category $\ho \Mod_{\Sigma^{\infty}_{+} (-)}$. The symmetric monoidal structure is given by the smash product $M \sma_{S} N$ over the sphere spectrum, with the diagonal action defined using the diagonal of $G$. More generally, the assignment $R \mapsto \ho \Mod_{R}$ defines a $\mathcal{H}\mathrm{opf}$-indexed symmetric monoidal category, where $\mathcal{H}\mathrm{opf}$ is the category of Hopf algebras in spectra, although we will not need that level of generality here.
\end{example}
 
An equivalence $\cM \simeq \cN$ of $S$-indexed symmetric monoidal categories consists of a pseudonatural transformation $\phi \colon \cM \arr \cN$ of pseudofunctors, each of whose components $\phi^{A} \colon \cM^{A} \arr \cN^{A}$ is an equivalence of categories.  More concretely, $\phi$ consist of a collection of strong symmetric monoidal functors $\phi^{A} \colon \cM^{A} \arr \cN^{A}$, each of which is an equivalence of categories, and natural monoidal isomorphisms 
\[
\xymatrix{
\cM^{B} \ar[d]_{f^*} \ar[r]^{\phi^{B}} \ar@{}[dr]|{\Downarrow\cong} & \cN^{B} \ar[d]^{f^*} \\
\cM^{A} \ar[r]^{\phi^{A}} & \cN^{A}
}
\]
%$\phi^{A} \circ f^* \arr f^* \circ \phi^{B}$ 
that are suitably compatible with the coherence isomorphisms for the base-change functors in $\cM$ and $\cN$. With this language, we may now state precisely how our Quillen equivalences interact with change of groups $G$.

% \begin{example}
% The assignment $B \mapsto \ho \Mod_{\Sigma^{\infty}_{+} \Omega B}$ defines a $\Top_{\mathrm{conn}}$-indexed symmetric monoidal category (problem: really defined on pointed spaces).  The symmetric monoidal structure is given by the smash prodcut $M \sma_{S} N$ over the sphere spectrum, with the diagonal action defined using the diagonal of $B$.  More generally, $R \mapsto \ho \Mod_{R}$ defines a $\mathrm{Hopf}$-indexed symmetric monoidal category, where $\mathrm{Hopf}$ is the category of Hopf algebras in spectra...
% \end{example}

\begin{theorem}\label{thm:equiv_indexed_symm_mon_cats}
The derived right adjoint $\bR \Phi \colon \ho \Sp_{BG} \arr \ho \Mod_{\Sigma^{\infty}_{+}G}$ associated to the Quillen equivalence $(\beta, \Phi)$ induces an equivalence 
\[
\ho \Sp_{B(-)} \simeq \ho \Mod_{\Sigma^{\infty}_{+} (-)}
\]
of $\TopGrp$-indexed symmetric monoidal categories.
\end{theorem}

\begin{proof}
When $G$ is the retract of a CW complex, the derived functor $\bR \Phi$ is an equivalence of categories by Corollary \ref{thm:equiv_G_monoid}. Let $[*]$ denote the canonical basepoint of $BG = |N_\cdot G|$. Then the derived functor is computed by applying $\Phi$ to a $qf$-fibrant spectrum, and so by Lemma \ref{lemm:rho_equiv}, we may lift the strong symmetric monoidal structure on the derived fiber functor $i_{[*]}^*$ \cite{MS}*{13.7.10} to the functor $\bR\Phi$.

In order to prove that $\bR \Phi$ is compatible with base-change functors, suppose we are given a continuous homomorphism $f \colon H \arr G$ and consider the induced diagram of spaces
\[
\xymatrix{
EH \ar[r]^{Ef} \ar[d] & EG \ar[d] \\
BH \ar[r]^{Bf} & BG.
}
\]
Let $Ef^{H\mathrm{-cof}} \colon EH^{H\mathrm{-cof}} \arr EG^{H\mathrm{-cof}}$ be a $qf$-cofibrant approximation of the homotopy equivalence $Ef$ in the projective model structure on $H$-spaces over $BG$.  Taking fiberwise suspension spectra, $Ef^{H\mathrm{-cof}}$ induces a stable equivalence of $qf$-cofibrant $H$-spectra over $BG$.  Notice that since $F_{BG}(-, -)$ is a Quillen bifunctor, if $M$ is a $qf$-fibrant spectrum over $BG$, then the spectrum $F_{BG}(EG^{H\mathrm{-cof}}_+, M)$ is a $q$-fibrant $H$-spectrum, and in particular agrees with the value of the derived restriction functor $f^* \colon \Mod_{\Sigma^{\infty}_{+}G} \arr \Mod_{\Sigma^{\infty}_{+}H}$ on $\Phi M = F_{BG}(EG^{G\mathrm{-cof}}_+, M)$.

We now pass to derived functors and abbreviate $E'H = EH^{H\mathrm{-cof}}$ and $E'G = EG^{G\mathrm{-cof}}$, as in the previous sections.  The derived pull-back functor 
\[
(Bf)^* \colon \ho \Sp_{BG} \arr \ho \Sp_{BH}
\]
has a left adjoint $(Bf)_{!}$ and the map $Ef$ induces a stable equivalence 
\[
(Bf)_{!}E'H _+\simeq E'G_+
\]
of spectra over $BG$.  It follows that there is a natural equivalence of derived functors
\begin{align*}
(\bR \Phi) (Bf)^* (M) &= F_{BH}(E'H_+, (Bf)^*M) \\
&\simeq F_{BG}((Bf)_! E'H_+, M) \\
& \simeq F_{BG}(E'G_+, M) = f^* (\bR \Phi) (M),
\end{align*}
and it is easily checked that this composite is a monoidal transformation and is compatible with the coherence isomorphisms $(g \circ f)^* \cong f^* \circ g^*$ for the base-change functors on $\ho \Sp_{B(-)}$ and $\ho \Mod_{\Sigma^{\infty}_{+}(-)}$.  

Finally, we may extend the result to arbitrary well-based group-like topological monoids $G$ by taking an approximation $G' \oarr{\simeq} G$ of $G$ by a group-like topological monoid that is a CW complex and composing with the induced equivalences of homotopy categories $\ho \Mod_{\Sigma^{\infty}_{+} G} \simeq \ho \Mod_{\Sigma^{\infty}_{+} G'}$ and $\ho \Sp_{BG} \simeq \ho \Sp_{BG'}$.  This completes the proof.
\end{proof}

This completes our analysis of how spectra over $BG$ for varying $G$ correspond to module spectra. Our final task is to relate this to the larger category of all topological spaces. One can make a module-theoretic description of the $\Top$-indexed symmetric monoidal category $\ho \Sp_{(-)}$ of parametrized spectra, using modules over ``ring spectra with many objects'', i.e. spectrally enriched categories. However, here we will be content to restrict attention to the category $\Top^{\conn}_*$ of pointed connected topological spaces.

%  (together with a whiskering to make the result well-based)
Any well-based variant of the Moore loop space defines a functor $\Omega \colon \Top^{\conn}_* \arr \TopGrp$. It induces an equivalence of homotopy categories, with inverse the classifying space functor $B \colon \TopGrp \arr \Top^{\conn}_*$.  There is a map of spaces $\xi \colon B \Omega X \arr X$ that is a weak homotopy equivalence when $X$ is connected \cite{May_class_fib}*{15.4}.  The derived base-change functor 
\[
\xi^* \colon \ho \Sp_{X} \arr \ho \Sp_{B \Omega X}
\]
is an equivalence of homotopy categories, and commutes with the derived base-change functor $f^*$ of parametrized spectra induced by a map $f \colon X \arr Y$ of pointed connected spaces.  This implies the following consequence of Theorem \ref{thm:equiv_indexed_symm_mon_cats}.

\begin{corollary}\label{cor:indexed_equiv}
The composite $\bR \Phi \circ \xi^*$  induces an equivalence 
\[
\ho \Sp_{(-)} \simeq \ho \Mod_{\Sigma^{\infty}_{+} \Omega(-)}
\]
of $\Top^{\conn}_*$-indexed symmetric monoidal categories.
\end{corollary}

\begin{remark}
Instead of working with homotopy categories, one may attempt to strengthen the theorem to the statement that the base-change functors on the actual categories agree up to coherent homotopy.  The work of Ando-Blumberg-Gepner \cite{ABG} using the language of $\infty$-categories is one way of making this idea precise.
\end{remark}

\begin{remark}\label{rem:free_loop}
By the uniqueness of adjoints, it follows that the equivalence of Corollary \ref{cor:indexed_equiv} commutes with the left and right adjoint $f_{!}$ and $f_*$ of $f^*$.  Let $\Delta \colon X \arr X \times X$ be the diagonal map of a space $X$, and consider the fiberwise suspension spectrum $(X, \Delta)_+ = \Sigma^{\infty}_{X \times X} (X, \Delta)_+$ of $X$ over $X \times X$.  Writing $r \colon X \arr *$ for the projection of $X$ to a point, one can compute that the derived base-change functors $r_{!} \Delta^* (X, \Delta)_+$ gives the stable homotopy type of the free loop space $LX = \Map(S^1, X)$.  On the other hand, $\bR \Phi (X, \Delta)_+$ is a model for the $\Sigma^{\infty} \Omega(X \times X)_+$-module spectrum $\Sigma^{\infty}_{+} \Omega X$.  Considering this as a bimodule, one finds that applying the derived base-change functors of module categories gives the THH of the spherical group ring:
\begin{align*}
r_! \Delta^* \bR \Phi (X, \Delta)_+ &\simeq \Sigma^{\infty}_{+} \Omega X \sma^{\bL}_{\Sigma^{\infty}_{+} \Omega X \sma \Sigma^{\infty}_{+} \Omega X^{\op}} \Sigma^{\infty}_{+} \Omega X \\
&\simeq \THH(\Sigma^{\infty}_{+} \Omega X; \Sigma^{\infty}_{+} \Omega X).
\end{align*}
Therefore, the equivalence of indexed symmetric monoidal categories proved in this paper recovers the equivalence $\THH(\Sigma^{\infty}_{+} \Omega X) \simeq LX_+$.  In the companion paper \cite{LM}*{\S5}, we generalize this to an equivalence of shadow functors on bicategories.  This comparison is a crucial component in our study of the transfer map of free loop spaces.
\end{remark}

\section{Costenoble-Waner dualizable spectra}\label{sec:cw_duality}

In this final section we prove Theorem \ref{thm:intro_cwdualizable}.  As background, we recall that May and Sigurdsson proved that every retract of a finite $qf$-cell spectrum in $\ho \Sp_{B}$ is Costenoble-Waner dualizable. %Although they use the language of $qf$-cells, their proof \cite{MS}*{18.2.6, 18.2.9} only uses the properties of $q$-cells. 
They also characterized the Costenoble-Waner dualizable spectra as the compact objects of the triangulated category $\ho \Sp_{B}$ \cite{MS}*{18.2.2}.

However, this left open the converse question of whether every Costenoble-Waner dualizable spectra is in fact a homotopy retract of a finite $qf$-cell spectrum.  We answer this question in the affirmative.  As a result, this excludes the possibility of exotic parametrized spectra which are dualizable but do not have finite presentations in terms of cells and attaching maps.  

As discussed in \cite{MS}*{18.2.10}, the essential difficulty is that in the parametrized setting, one cannot simultaneously make a space both compact and fibrant. We will circumvent this by taking our compact objects in the category of module spectra, before applying the Borel construction to get parametrized spectra.

In fact, our methods work equally well to characterize all dualizable spectra in the bicategory $\Ex$ of spectra over varying base spaces.  Suppose that a spectrum $X$ over $A \times B$ has the property that every derived pullback $i_{a}^* X$ along the inclusion of a point $a \in A$ is a retract of a finite $qf$-cell spectrum in $\ho \Sp_{B}$.  The argument of May and Sigurdsson referenced above actually proves that $X \in \ho \Sp_{A \times B} = \Ex(A, B)$ is dualizable over $B$ as a 1-cell in the bicategory $\Ex$.  See the companion paper \cite{LM}*{\S4, 5} and the references \citelist{ \cite{MS} \cite{PS12}} for more details on duality in the bicategories $\Ex$ and $\Bimod$. We will now prove the converse.

Suppose that $X \in \Ex(A, B)$ is dualizable over $B$.  In other words, $X$ is a spectrum over $A \times B$ and we are given another spectrum $Y \in \Ex(B, A) = \ho \Sp_{B \times A}$ and morphisms
\[
\coev \colon U_{A} \arr X \odot_{B} Y \qquad \eval \colon Y \odot_{A} X \arr U_{B}
\]
in the homotopy categories of spectra over $A \times A$ and $B \times B$ satisfying the triangle identities for a duality.  Here, the bifunctor $X \odot_{B} Y = \pi^{B}_{!} \Delta_{B}^* (X \osma Y)$ is the horizontal composition in the bicategory $\Ex$ \cite{MS}*{\S17}.  Specializing to the case $A = \ast$, we have the notion of a Costenoble-Waner dualizable spectrum $X$ in $\ho \Sp_{B}$.  

We assume without loss of generality that $A$ and $B$ are connected and pointed.  For the sake of clarity, let us write
\[
\Phi_{X} \colon \ho \Sp_{X} \arr \ho \Sp_{B\Omega X} \arr \ho \Mod_{\Sigma^{\infty}_{+} \Omega X}
\]
 for the composite equivalence from Corollary \ref{cor:indexed_equiv}.  The morphisms $\Phi_{A \times A}(\coev)$ and $\Phi_{B \times B}(\eval)$ exhibit $\Phi_{A \times B}X$ and $\Phi_{B \times A} Y$ as a dual pair in the bicategory $\Bimod$.  In particular, forgetting the action of $\Omega A$, the spectrum $\Phi_{A \times B} X$ is dualizable as a $\Sigma^{\infty}_{+} \Omega B$-module, so it is a retract of a finite cell spectrum in $\ho \Mod_{\Sigma^{\infty}_{+}\Omega B}$.  

It now suffices to prove that the inverse of the equivalence $\Phi_{B}$ takes finite $\Sigma^{\infty}_{+}\Omega B$-cell spectra to finite $qf$-cell spectra over $B$.
Without loss of generality, we may work in $\ho \Sp_{BG}$, where $G$ is a topological group.
%~ We make this additional assumption because we will need to use the fact that the projection $\pi \colon EG \arr BG$ is an $h$-fibration.
In other words, Theorem \ref{thm:intro_cwdualizable} follows from:
\begin{proposition}
If $X$ is a finite $\Sigma^{\infty}_{+}G$-cell spectrum, then $\beta X$ is equivalent in $\ho \Sp_{BG}$ to a finite $qf$-cell spectrum.
\end{proposition}
\begin{proof}
After sufficient desuspensions, we may assume that that $X$ is the suspension spectrum of a finite free $G$-cell complex of pointed spaces, which we also denote by $X$.  By Lemma \ref{lemm:cofibrant_detection}, it suffices to prove the result for the functor $X \longmapsto EG_+ \osma_{G} X$ instead of $\beta$.  Write $X$ as the colimit of equivariant skeleta $X^{(i)}$ where $X^{(i+1)}$ is obtained from $X^{(i)}$ by attaching a single cell of the form $G \times D^n$. We inductively construct a finite $qf$-cell complex $Y_i$ over $BG$ with a weak homotopy equivalence $f_i: Y_i \arr EG_+ \osma_{G} X^{(i)}$ of ex-spaces over $BG$. The base case is $Y_0 = BG = EG_+ \osma_{G} X^{(0)}$.  For the inductive step, we apply $EG_+ \osma_{G}(-)$ to the pushout which constructs $X^{(i + 1)}$ from $X^{(i)}$, and build around it the following diagram of ex-spaces.
\[ \xymatrix{
&&& Y_i \ar[d]^-\simeq_-{f_i} \\
S^{n-1}_+ \ar[r]_-{(*, \id)} \ar@{-->}[rrru]^-{\alpha} & (EG \times S^{n-1})_+ \ar[r]^-\cong \ar[d] & EG_+ \osma_{G} (G \times S^{n-1})_+ \ar[r] \ar[d] & EG_+ \osma_{G} X^{(i)} \ar[d] \\
& (EG \times D^n)_+ \ar[r]^-\cong & EG_+ \osma_{G} (G \times D^n)_+ \ar[r] & EG_+ \osma_{G} X^{(i+1)}
} \]
On the left, we include $S^{n - 1}$ into $EG \times S^{n - 1}$ using the basepoint $* \in EG$.  The weak equivalence $f_i$ is provided by the inductive hypothesis. We observe that $Y_i$ need not be fibrant over $BG$, so $f_i$ may not have an inverse in the homotopy category of spaces over $BG$, only in the homotopy category of spaces.  The latter maneuver gives a map $\alpha \colon S^{n - 1} \arr Y_i$ as in the diagram and a homotopy between the two routes in the upper triangle. This gives a diagram
\[
\xymatrix{
D^{n} \ar[d]_{(*, \id)}^-\simeq & S^{n - 1} \ar[d]_{(*, \id)}^-\simeq \ar[l]_{i} \ar[r]^-{\alpha} & Y_i \ar[d]^{\simeq}_-{f_i} \\
EG \times D^n & EG \times S^{n - 1} \ar[l]_-{\id \times i} \ar[r] & EG_+ \osma_{G} X^{(i)}
}
\]
in which the left square commutes and the right square commutes up to a given homotopy. This in turn induces a weak equivalence of spaces
\[
D^{n} \cup^{h}_{S^{n - 1}} Y_i \arr EG_+ \osma_{G} X^{(i + 1)}
\]
from the double mapping cylinder of the top row to the pushout of the bottom row.  We insist that the induced map uses the given homotopy at double-speed along the first half of the mapping cylinder, then is constant at $\alpha$ for the second half of the mapping cylinder.  Let us denote this new map by $f_{i+1}$, and denote its source by $Y_{i+1}$.  We define the inclusion $BG \arr Y_{i+1}$ by composing the inclusion $BG \arr Y_i$ with the inclusion of one of the ends of the mapping cylinder. We define the projection $Y_{i+1} \arr BG$ by composing $f_{i+1}$ with the projection of $EG_+ \osma_{G} X^{(i + 1)}$. Observe that our requirement on how $f_{i + 1}$ is defined in terms of the given homotopy to $\alpha$ makes the cell we have just attached to $Y_{i}$ a $qf$-cell.  In effect, this amounts to composing the projection map $D^n \arr BG$ of the new cell for $Y_{i+1}$ with 
%inclusion $S^{n - 1} \arr D^n$ into the cell created by the mapping cylinder with
the expansion by two map
\begin{align*}
D^n &\arr D^n \\
x &\longmapsto \begin{cases}
2x \quad &\text{if $\abs{x} \leq \tfrac{1}{2}$} \\
x/\abs{x} \quad &\text{if $\abs{x} > \tfrac{1}{2}$},
\end{cases}
\end{align*}
and this makes the resulting disk an $f$-disk over $BG$.
The map $f_{i+1}$ is a weak equivalence of ex-spaces with source a finite $qf$-cell, completing the inductive step of the proof.
\end{proof}

We note that it is a consequence of our proof that an object of $\ho \Sp_{B}$ is equivalent to a finite $q$-cell spectrum if and only if it is equivalent to a finite $qf$-cell spectrum.  Recall that Waldhausen's functor $A(B)$ may be described as the $K$-theory of the Waldhausen category of retractive spaces over $B$ that are retracts of finite $q$-cell complexes.  Since $K$-theory is invariant under stabilization of Waldhausen categories, $A$-theory may equivalently be described as the $K$-theory of parametrized finite $q$-cell spectra over $B$, and thus also as the $K$-theory of parametrized finite $qf$-cell spectra over $B$.  Theorem \ref{thm:intro_cwdualizable} then implies the interpretation of Waldhausen's $A$-theory stated in Corollary \ref{cor}.

\end{document}